\documentclass[12pt,a4paper,oneside]{amsart}

\usepackage{amsfonts, amsmath, amssymb, amsthm, amscd}
\usepackage{hyperref}
\hypersetup{
  colorlinks   = true, 
  urlcolor     = blue, 
  linkcolor    = blue, 
  citecolor   = red 
}
\usepackage{anysize}
\marginsize{2cm}{2cm}{1.5cm}{1.5cm}
\usepackage[pdftex]{graphicx}

\newtheorem{theorem}{Theorem}[section]
\newtheorem{lemma}[theorem]{Lemma}

\newtheorem{corollary}[theorem]{Corollary}

\theoremstyle{definition}

\theoremstyle{remark}
\newtheorem{remark}[theorem]{Remark}

\newtheorem{example}[theorem]{Example}

\numberwithin{equation}{section}

\newcommand{\where}{\mathop{\ |\ }\nolimits}

\DeclareMathOperator{\dist}{dist}

\renewcommand{\epsilon}{\varepsilon}

\renewcommand{\phi}{\varphi}

\newcounter{fig}

\title{Envy-free division using mapping degree}

\author{Sergey Avvakumov{$^\spadesuit$}}

\author{Roman~Karasev{$^\clubsuit$}}

\thanks{{$^\spadesuit$} Has received funding from the Austrian Science Fund (FWF), Project P31312-N35, and the European Research Council under the European Union's Seventh Framework Programme ERC Grant agreement ERC StG 716424 -- CASe}

\thanks{{$^\clubsuit$} Supported by the Federal professorship program grant 1.456.2016/1.4 and the Russian Foundation for Basic Research grants 18-01-00036 and 19-01-00169}

\address{Sergey~Avvakumov, Department of Mathematical Sciences, University of Copenhagen, Universitetspark 5, 2100 Copenhagen, Denmark}
\email{savvakumov@gmail.com}

\address{Roman~Karasev, Moscow Institute of Physics and Technology, Institutskiy per. 9, Dolgoprudny, Russia 141700\newline \indent
Institute for Information Transmission Problems RAS, Bolshoy Karetny per. 19, Moscow, Russia 127994}
\email{r\_n\_karasev@mail.ru}
\urladdr{http://www.rkarasev.ru/en/}

\subjclass[2010]{51F99, 52C35, 55M20, 55M35}
\keywords{Fair partitions, Envy-free divisions, Knaster--Kuratowski--Mazurkiewicz theorem, Equivariant mapping degree}

\begin{document}

\begin{abstract}
In this paper we study envy-free division problems. The classical approach to such problems, used by David Gale, reduces to considering continuous maps of a simplex to itself and finding sufficient conditions for this map to hit the center of the simplex. The mere continuity of the map is not sufficient for reaching such a conclusion. Classically, one makes additional assumptions on the behavior of the map on the boundary of the simplex (for example, in the Knaster--Kuratowski--Mazurkiewicz and the Gale theorem). 

We follow Erel Segal-Halevi, Fr\'ed\'eric Meunier, and Shira Zerbib, and replace the boundary condition by another assumption, which has the meaning in economy as the possibility for a player to prefer an empty part in the segment partition problem. We solve the problem positively when $n$, the number of players that divide the segment, is a prime power, and we provide counterexamples for every $n$ which is not a prime power. We also provide counterexamples relevant to a wider class of fair or envy-free division problems when $n$ is odd and not a prime power.

\emph{In this arxiv version that appears after the official publication we have corrected the statement and the proof of Lemma 3.4.}
\end{abstract}

\maketitle

\section{Introduction}

In this paper we study a classical game theoretic problem: $n$ players want to divide a resource among themselves. Is it always possible to do so in a fair, in some sense, way?

We consider a simple case, when the resource is the line segment $[0,1]$, and allow its partitions into $n$ closed (possibly degenerating to empty) segments with pairwise disjoint interiors. For each partition of $[0,1]$, each of the players would be satisfied to take one of the partition pieces, the choice of a player need not be unique. As a simple example, every player may rate the pieces with her/his own integrable ``value'' function $f_i$ on $[0,1]$, and prefer any of those partition segments that maximize the value of the integral of $f_i$ over the partition segment. However, we do not assume that the players have such ``value'' functions; in fact, they may rate the segments of the partition using an arbitrary complicated logic or no logic at all.

A partition (of the segment) is called \emph{envy-free}, if the players can be matched with the pieces of the partition so that each player is satisfied with the matching piece. Following Gale \cite{gale1984} and other classical results, we also make a natural ``continuity'' assumption: A player prefers the $i$th piece if in another, but arbitrarily close to the given partition configuration, she/he also prefers the $i$th piece.

One may additionally assume that the players are never satisfied with the empty pieces of any partition, this is the so called ``something is better than nothing'' assumption. In other words, any player prefers nonempty parts over empty parts. Assuming that ``something is better than nothing'', the existence of an envy-free segment partition is guaranteed by Gale's theorem (see Theorem \ref{theorem:gale} below for the precise statement).

Without the ``something is better than nothing'' assumption the situation becomes somewhat more complicated. In terms of the original economy problem, we may be considering the resource which comes with an additional cost. For some partitions, the cost of every nonempty piece might exceed its value, in which case a player might prefer to take an empty piece instead.

For the segment partitioning problem without the ``something is better than nothing'' assumption, in \cite{segal2018} it was proved that envy-free segment partitions exist for $n=3$ (the case $n=2$ is an easy exercise). In \cite{meunier-zerbib2018} the result was extended to $n=4$, or any prime $n$. In this work we give a complete solution to the same problem: \textit{If $n$ is a prime power then an envy-free segment partitioning with the possibility to choose the empty part always exists (equivalent to Theorem~\ref{theorem:prime-power}). Conversely, for every $n$ which is not a prime power, there exists an instance of this problem with no solution (equivalent to Theorem~\ref{theorem:n-even}).}

We need some preparations and setting up the notation in order to give mathematically precise statements of our results; in the introduction we only give informal statements. The rest of the paper is organized as follows. In Section \ref{section:classical} we outline the classical results and reductions of the envy-free division problems to precise mathematical questions. We start from the mapping version of the Knaster--Kuratowski--Mazurkiewicz theorem, Theorem \ref{theorem:kkm}, and then proceed to Gale's theorem, Theorem~\ref{theorem:gale}. In Section \ref{section:easy} we review some easy results that prepare the reader to understanding the substantially new results in subsequent sections. 

For classical results in Section \ref{section:classical} and for new results in Section \ref{section:segment} we emphasize that the natural way to handle the envy-free segment partition problem is to analyze necessary and sufficient conditions that a continuous map of a simplex to itself hits its center; which amounts to determining possible mapping degrees of maps between spheres under some additional assumptions, analogous to \emph{equivariance} with respect to a group action. 

In Section \ref{section:borsuk-ulam} we show another fundamental result: \textit{For $n$ odd and not a prime power, there is no Borsuk--Ulam theorem for equivariant maps from a Hausdorff compactum $X$ with a free action of the permutation group $\mathfrak S_n$ to $\mathbb R^n$ with the permutation action of $\mathfrak S_n$.} This result is not related to the original segment envy-free division problem, but it prevents using some of the well-known general techniques for other envy-free division or fair partition problems. Its analogues and their consequences are developed in \cite{avku2019,aks2019}.


\subsection*{Acknowledgments}
The authors thank Shira~Zerbib, Fr\'ed\'eric~Meunier, Alfredo~Hubard, Oleg~Musin, Arkadiy~Skopenkov, Peter~Landweber, Pavle Blagojevi\'c, and the unknown referee for useful remarks and corrections to the text.

\section{Classical KKM-type results and partition problems}
\label{section:classical}

We recall some classical results around the Knaster--Kuratowski--Mazurkiewicz theorem~\cite{kkm1929} with modifications from \cite{gale1984,bapat1989}. Let $\Delta^{n-1}$ be the $(n-1)$-dimensional simplex, which we usually parametrize as
\[
\Delta^{n-1} = \left\{(t_1,\ldots,t_n)\in\mathbb R^n \where  t_1,\ldots, t_n \ge 0,\ t_1 + \dots + t_n = 1 \right\}.
\]
We also denote by $\Delta^{n-1}_i$ the facet of $\Delta^{n-1}$ given by the additional constraint $t_i=0$. Sometimes, when we know the dimension $n$, we will denote the simplex by $\Delta$ and its facets by $\Delta_i$.

In the above notation the KKM theorem reads: \textit{Let $A_1,\ldots, A_n$ be closed subsets of $\Delta^{n-1}$, covering the simplex, such that for every $i=1,\ldots,n$ the intersection $\Delta^{n-1}_i\cap A_i$ is empty. Then the intersection $A_1\cap A_2\cap \dots \cap A_n$ is not empty.} We will also use the KKM theorem in the mapping form:

\begin{theorem}[The mapping KKM theorem]
\label{theorem:kkm}
Assume $f : \Delta^{n-1} \to \Delta^{n-1}$ is a continuous map such that for all $i$ we have $f(\Delta^{n-1}_i) \subset \Delta^{n-1}_i$. Then $f$ is surjective.
\end{theorem}
\begin{proof}
Let us approximate $f$ with a PL map in order to treat the mapping degree geometrically. A PL map assumes a subdivision of $\Delta$, in order to refer to the faces of the original (tautological) triangulation of $\Delta$ we use the term \emph{big faces}.

We may assume that the approximating PL map has the same property that any big facet (and hence any big face of arbitrary dimension) is mapped to itself. Considering $\Delta$ as a PL manifold with boundary, we notice that $f$ takes boundary to the boundary. Therefore the mapping degree of $f$ is well defined and is equal to the mapping degree of its restriction $f|_{\partial \Delta} : \partial\Delta\to\partial\Delta$. This is clear either from the geometric definition of the mapping degree, or from the exact homology sequence of the pair $(\Delta, \partial\Delta)$ and action of $f_*$ on it, or from the Stokes theorem for differential forms in $\Omega^{n-1}(\Delta)$ (this may sound strange at this point, but we will essentially use differential forms in the proof of Theorem~\ref{theorem:prime-power} below).

Then we prove by induction on the dimension that the mapping degree of $f$ equals $1$. The case of dimension $n=1$ is clear, for the induction step we note that the restriction to a big facet, $f|_{\Delta_i} : \Delta_i\to\Delta_i$, satisfies the same assumptions (big faces go to themselves) and hence we conclude that the degree of $f|_{\Delta_i}$ equals $1$. From the geometric description of the degree of a PL map, this degree is the same as the degree of $f|_{\partial \Delta}$, which in turn equals the degree of $f$.
\end{proof}

\begin{proof}[Reduction of the classical KKM to its mapping version]
Replace $A_i$ with a continuous function $g_i : \Delta \to\mathbb R$, such that $g_i(A_i)=1$ and $g_i(x)=0$ for $x$ outside an $\epsilon$-neighborhood of $A_i$. When $\epsilon>0$ is sufficiently small, we will have $g_i(\Delta_i) = 0$ from the assumption $\Delta_i\cap A_i = \emptyset$.

Since the $A_1,\ldots,A_n$ cover the simplex, we conclude that $g_1(x) + \dots + g_n(x) \ge 1 > 0$ for every $x\in\Delta$. Dividing every $g_i$ by this sum, we obtain non-negative continuous functions $f_1,\ldots,f_n$ with unit sum everywhere in the simplex. Such $f_i$ are coordinates of a map
\[
f : \Delta \to \Delta,
\]
and the property $f_i(\Delta_i) = 0$ means that any facet goes to itself. Hence by the mapping KKM theorem $f$ is surjective and therefore there exists $x\in\Delta$ such that $f_i(x)=1/n$ for any $i$. Such a point $x$ is in the $\epsilon$-neighborhood of $A_i$ for every $i$. Passing to the limit as $\epsilon\to 0$ and using compactness of $\Delta$ we assume that $x$ tends to a point in $\Delta$. From the fact that all $A_i$ are closed we conclude that this limit point belongs to $A_1\cap \dots \cap A_n$ and show that this intersection is not empty.
\end{proof}

Now we proceed to the generalization of the KKM theorem, useful in proving existence of equilibria in  questions relevant to economy.

\begin{theorem}[Gale's theorem]
\label{theorem:gale}
Let $A_{ij}$ be closed subsets of $\Delta^{n-1}$, indexed by $i=1,\ldots,n$ and $j=1,\ldots,n$. Assume that for every fixed $j$ the family of sets $\{ A_{1j}, A_{2j},\ldots, A_{nj} \}$ covers the simplex, and $A_{ij}\cap \Delta^{n-1}_i$ is empty for every $i$ and $j$. Then there exists a permutation $\sigma$ of size $n$ such that the intersection $A_{1\sigma(1)}\cap A_{2\sigma(2)}\cap\dots\cap A_{n\sigma(n)}$ is not empty.
\end{theorem}

\begin{proof}
We essentially reproduce the (sketch of the) proof in \cite[Proof of the lemma on page 63]{gale1984}, giving more details. Replace each set $A_{ij}$ by a function $g_{ij}$. Using the covering assumption, we may normalize $g_{ij}$ to obtain $f_{ij}$ such that
\[
f_{1j} + \dots + f_{nj} = 1
\]
at any point of the simplex and any $j$, and also $f_{ij}(\Delta_i)=0$. Now introduce non-negative functions
\[
h_i = \frac{f_{i1} + \dots + f_{in}}{n},
\]
which still satisfy $h_1 + \dots + h_n = 1$ everywhere in the simplex, and $h_i(\Delta_i)=0$. Hence there appears a continuous map $h : \Delta\to\Delta$ sending each facet to itself and by the mapping KKM theorem we conclude that there exists $x\in \Delta$ such that $h_i(x) = 1/n$ for every $i$.

Evaluating our original matrix of functions $f_{ij}$ at the point $x$, we conclude that 
\[
\sum_i f_{ij}(x) = 1,\quad \sum_j f_{ij}(x) = 1.
\]
The matrix $(f_{ij}(x))$ is doubly stochastic and the Birkhoff--von Neumann theorem \cite{birkhoff1946} (see also the textbook \cite[pages 56--58]{barvinok2002}) asserts that this matrix is a convex combination of permutation matrices. In particular, there exists a permutation $\sigma$ such that $f_{i\sigma(i)}(x) > 0$ for every $i$. Alternatively, this can also be deduced with a little effort from Hall's marriage theorem \cite[Theorem~1]{hall1935}. Going to the limit and using the compactness and closeness again, we obtain that $A_{1\sigma(1)}\cap\dots\cap A_{n\sigma(n)}\neq\emptyset$.
\end{proof}

For far-reaching generalizations of these theorems, see for example \cite[Theorem 3.1]{musin2017}, which provides a Gale-type theorem corresponding to homotopy classes of maps from topological spaces to spheres, of which the degree of a map between spheres of equal dimensions is a particular case.

The meaning of Gale's theorem in economy can be illustrated as follows. The simplex $\Delta^{n-1}$ (sometimes) parametrizes partitions of a resource into $n$ parts. The set $A_{ij}$ corresponds to the partitions where the player $j$ would be satisfied to take the $i$th part of the resource and leave the rest to the other players. The other assumptions of the theorem mean that in every partition every player would be satisfied with some part, and nobody will be satisfied to take the empty part with $t_i=0$. The conclusion of the theorem then means that there exists a partition and an assignment $\sigma$ of the parts to the players such that every player will be satisfied.

The basic case that we mostly study below is when the simplex $\Delta^{n-1}$ parametrizes partitions of a segment $[0,1]$ into parts
\[
[0, t_1], [t_1, t_1 + t_2],\dots, [t_1+\dots + t_{n-1}, 1].
\]
The facet $\Delta_i$ then corresponds to the situation when $t_i=0$ and hence the $i$th partition segment degenerates to one point. We will identify such a one point segment with the empty set in the subsequent sections.

\section{When some players may choose nothing}
\label{section:easy}

\subsection{Assume that some parts may be dropped}

What happens when $A_{ij}\cap \Delta_i$ is non-empty in Gale's theorem, or, in terms of the envy-free segment partition problem, if some players sometimes prefer to take nothing from the resource partition? This question was left as an exercise to the reader in \cite[middle of page 3]{meunier-zerbib2018}, let us perform this exercise here.

We may obtain a result about this by adjusting the situation to the assumption of Gale's theorem. Let us remove from $A_{ij}$ the part where $t_i < \epsilon$. This will satisfy the assumption $A_{ij}\cap \Delta_i=\emptyset$ of Gale's theorem, but will break the assumption that $\{A_{ij}\}_{i=1}^n$ cover the simplex for every $j$.

In order to fix the covering assumption, given $j$, let us add $t\in \Delta$, which did not belong to any $A_{ij}$, to $A_{i_{max}j}$ where $t_{i_{max}}$ is a maximal coordinate of the point $t$, there may be several maximal coordinates. Such a modification of $A_{ij}$ keeps the assumption that the coordinate $t_i$ is no smaller than $\epsilon$ on $A_{ij}$.

Now apply Gale's theorem to the modified sets to obtain a permutation $\sigma$ and a point $x_\epsilon\in \bigcap A_{i\sigma(i)}$. If all the coordinates of $x_\epsilon$ are greater than $\epsilon$ then we are in the range where we did not modify anything and the problem is solved. 

Otherwise there exist coordinates of $x_\epsilon$ that are at most $\epsilon$. In this case we are going to the limit $\epsilon\to+0$, from the compactness we may assume that $x_\epsilon\to x$ and the permutation is all the time the same. In the coordinates $x_1,\ldots,x_n$ of the limit configuration some coordinates $x_i$ will then be zero, otherwise we are in the first case. 

In this limit configuration, speaking in terms of the envy-free segment partition problem, some player $j=\sigma(i)$ may be dissatisfied with the assignment of the part $i$ to her/him. But this may only happen in the situation when this player preferred parts with some $t_{i'} < \epsilon$ in the neighborhood of $x$, we may assume $i'$ fixed here. By the closedness of the preference set $A_{i'j}$ we obtain that $x_{i'}=0$ for the limit point $x$ and that the player $j$ does prefer the emptyset in the partition $x$.  

Now we conclude:

\begin{corollary}
Under the assumptions of Gale's theorem, modified so that some players may sometimes prefer nothing, it is possible to find a partition, assign some parts to the players, drop some unwanted parts, and assign nothing to some of the players, so that all players will be satisfied.
\end{corollary}

\subsection{General observations when no part may be dropped}

In our argument it is crucial that whenever the player is satisfied with the part $i$ such that $t_i=0$, he/she will also be satisfied with any other part $i'$ such that $t_{i'}=0$. In other words, there is only one sort of ``nothing''.

Now we return to the setting when it is not allowed to drop parts in a partition. Let us explain why any problem of KKM--Gale type is roughly equivalent to the study of continuous maps $f : \Delta^{n-1}\to\Delta^{n-1}$. We will always use the covering assumption, in terms of the envy-free segment partition problem, in every partition any player is satisfied with some of the parts.

In one direction, we start from the preference sets $A_{ij}$ and pass to functions $f_{ij}$, as in the proof of Theorem~\ref{theorem:gale} above. If certain assumptions on $A_{ij}$ imply certain other assumptions on $f_{ij}$ that, in turn, allow us to conclude that the map hits the center of the simplex, then we are done by essentially the same argument.

In the other direction, having a continuous map $f : \Delta^{n-1}\to \Delta^{n-1}$, we put 
\[
A_{ij} = \left\{t\in \Delta^{n-1} \where \forall i'\ f_i(t)\ge f_{i'}(t) \right\}.
\]
This definition does not depend on $j$, that is the players have precisely the same preference, hence we put $A_i = A_{ij}$. The family of closed sets $A_1,\ldots, A_n$ covers the simplex. Note that in the case, when all the players have the same preference, the setting of Gale's theorem degenerates to the setting of the KKM theorem. Now we observe that the $A_i$ have a common point if and only if
\[
f_1(t) = \dots = f_n(t) = \frac{1}{n}
\] 
for some $t$.

Since it is easy to build a continuous map $f:\Delta^{n-1}\to\Delta^{n-1}$ missing the center of the simplex, it is now clear that in order to have a Gale-type theorem, we need some assumption like ``no player is satisfied with an empty part''. Here we give a very explicit example:

\begin{example}
One may ask if it is sufficient to have the assumption ``if somebody prefers nothing then he/she does not care on which position this nothing occurs'' and prove a KKM--Gale-type theorem, without using any equivariance assumptions or other similar assumptions. This is not the case already for the KKM theorem. Take the triangle $\Delta^2$ and put
\[
A_1=\Delta^2,\quad A_2=\{t_1=t_2=0\},\quad  A_3=\{t_1=t_3=0\}.
\]
In terms of the envy-free segment partition problem, in all cases the player prefers part $1$. When parts $1$ and $2$ are empty, the player also prefers part $2$. When parts $1$ and $3$ are empty, the player also prefers part $3$. But there is no configuration where the player prefers all three parts; or in case of Gale's theorem, where the preferences of three identical players are met.
\end{example}

\subsection{Using permutation equivariance}

One possible way is to introduce an assumption of ``equivariance on the boundary'' with respect to the action of the permutation group $\mathfrak S_n$ on the simplex $\Delta^{n-1}$ by permuting the coordinates. For example, in Gale's theorem we may require, for every $i,j=1,\ldots,n$ and any permutation $\sigma$,
\[
\sigma A_{ij}\cap \partial \Delta^{n-1}= A_{\sigma(i)j} \cap \partial \Delta^{n-1}.
\]

In terms of the envy-free segment partition problem, this means that when a partition has empty parts (the boundary of the simplex) and the parts of a partition are permuted, then the players trace the parts they prefer and continue preferring them. When a partition has $n$ non-empty parts, then the players may take the order into account. Perhaps, the formulation here is not very natural from the point of view of economy, but it may serve to us as a mathematically natural example, which we can handle. Here we give a positive result for this setting:

\begin{theorem}
\label{theorem:equivariant-gale}
The KKM theorem and Gale's theorem are valid when it is allowed to choose empty parts if we impose the ``equivariance on the boundary'' assumption and also assume that $n$ is a prime power.
\end{theorem}

This theorem follows from well-known results on degrees of equivariant maps between spheres, see for example \cite{marzantowicz1989} or the textbook \cite[Sketch of the proof of Theorem 6.2.5]{matousek2003using}. The technique of the latter reference shows that the homological trace of a $G$- equivariant map $f : S^{n-1}\to S^{n-1}$ is divisible by a prime $p$ (and hence its degree is $\pm 1$ modulo $p$) if all the $G$-orbits in $S^{n-1}$ have size divisible by $p$. Here we provide a similar explicit argument proving this theorem, because we will use modifications of this argument to establish further results. In particular, Theorem~\ref{theorem:n-odd} asserts that dropping the assumption that $n$ is a prime power, at least for odd $n$, leads to an opposite conclusion.

\begin{lemma}
\label{lemma:finite-to-one}
Assume $G$ is a finite group acting on a polyhedron $P$ and acting linearly on a vector space $V$. Assume that for any subgroup $H\subseteq G$ the inequality $\dim P^H \le \dim V^H$ holds for the subspaces of $H$-fixed points. Then for any $G$-invariant triangulation of $P$ its second barycentric subdivision has the following property: The set of $G$-equivariant PL maps $f : P \to V$, linear on faces of the second barycentric subdivision, has an open dense subset consisting of maps with finite fibers $f^{-1}(y)$ for any $y\in V$. 
\end{lemma}

\begin{proof}
Let us first make one barycentric subdivision. The vertices of the barycentric subdivision are marked by the dimension of the faces they originate from and those marks are preserved by $G$. Hence the action of $G$ has the following property: \emph{For any $g\in G$ and any face $\phi$ we have $g(\phi) = \phi$ if and only if $g$ is the identity map on $\phi$.} We now assume that the triangulation of $P$ has this property. 

Now consider $G$-equivariant maps, linear on faces of the barycentric subdivision $P'$ (this may be the second barycentric subdivision we make). We show that a dense open subset of such maps (that is a \emph{generic map} of this kind) has the required property. Such a map $f : P \to V$ is defined whenever we define it equivariantly on vertices of the subdivision $P'$, and we argue by induction on the poset of the vertices of $P'$, which is the same as the poset of faces of $P$.

Assume we have a vertex $\phi\in P'$ and consider possible values $f(\phi)$. Let $H$ be the stabilizer of $\phi$ as a vertex of $P'$, by our assumption in the beginning ofthe proof this is also the stabilizer of every point of $\phi$ as a face of $P$. The value $f(\phi)$ must be chosen in $V^H$ and $f(\phi)\in V^H$ is the only constraint needed to extend $f$ to the orbit $G\phi$ equivariantly. For any face of $P'$, given by a chain of vertices of $P'$
\[
\phi_1 < \phi_2 < \dots < \phi_k < \phi
\]
of faces of $P$, we assume by induction that $f(\phi_1), \ldots, f(\phi_k)$ are affinely independent and form a $(k-1)$-dimensional simplex in $V$ (otherwise $f$ is not finite-to-one on $\phi_k$). 

The dimension assumption in the statement of the lemma means that $k \le \dim \phi \le \dim V^H$ (speaking of dimension, we consider $\phi$ as a face of $P$ and note that $\phi\subseteq P^H$), hence for a generic choice of $f(\phi)\in V^H$ the points $f(\phi_1), \ldots, f(\phi_k), f(\phi)$ are affinely independent. This applies to all chains that end in $\phi$ and completes the induction step and the proof is complete.
\end{proof}

\begin{proof}[Proof of Theorem \ref{theorem:equivariant-gale}]
Consider any $\mathfrak S_n$ equivariant map $\partial \Delta^{n-1} \to \partial \Delta^{n-1}$ and compose it with the inclusion $\partial \Delta^{n-1}\subset W_n$ into the affine span of $\Delta^{n-1}$ to obtain a $\mathfrak S_n$-equivariant map
\[
f_1 : \partial \Delta^{n-1} \to W_n.
\]
Let $f_0 : \partial \Delta^{n-1}\to W_n$ be the standard $\mathfrak S_n$-equivariant inclusion. Connect them by an equivariant homotopy 
\[
h : \partial \Delta^{n-1}\times [0,1] \to W_n,
\]
which can be chosen as $h(x,t) = (1-t) f_0(x) + t f_1(x)$. 

Note that the difference in the degrees of $f_0$ and $f_1$ as maps of $\partial \Delta^{n-1}$ to itself equals the degree of $h$ over the center $c\in \Delta^{n-1}$, which may be considered as the origin $0\in W_n$. This follows from the fact that the degree of a map between closed connected oriented manifolds with boundary $h : M\to N$ satisfying $h(\partial M)\subset \partial N$ is well defined and equals the degree of the restriction $h|_{\partial M} : \partial M\to \partial N$. Here $M=\partial \Delta^{n-1}\times [0,1]$ and $N=\Delta^{n-1}$.

Lemma \ref{lemma:finite-to-one} applies because 
\[
\left( \partial \Delta^{n-1} \times [0,1] \right)^H = \left( \partial \Delta^{n-1} \right)^H \times [0,1],
\]
it allows us to assume, after a perturbation of $h$, that $h^{-1}(0)$ is finite and the degree can be counted geometrically as the sum of local degrees at the points $x\in h^{-1}(0)$. The degree at a point $x\in\partial \Delta^{n-1}$ equals to the degree at any other point $\sigma x$ for $\sigma\in\mathfrak S_n$, because $\sigma$ acts of the orientation of the domain and the range by the permutation sign. 

Hence we are interested in the size of the orbit of a point $x$, which is counted as follows: Split the barycentric coordinates of $x$ into blocks of equal coordinates, let $k_1,\ldots,k_\ell$ be the sizes of the blocks, note that for the boundary points $x$ we have at least two blocks. Then the stabilizer of $x$ has size $k_1!\cdots k_\ell!$ and the size of the orbit is
\[
\frac{n!}{k_1! \cdots k_\ell!} = \binom{n}{k_1\ k_2\ \cdots\ k_\ell}.
\]
Since the multinomial coefficient is the product of the binomial coefficients
\[
\binom{n}{k_1\ k_2\ \cdots\ k_\ell} = \binom{n}{k_1}\cdot \binom{n-k_1}{k_2}\dots \binom{n-k_1 - \dots - k_{\ell - 1}}{k_\ell},
\]
the Lucas theorem \cite{lucas1878} on divisibility of the binomial coefficients by primes implies that the size of the orbit is divisible by $p$ when $n=p^\alpha$, because the first factor in the above formula is already divisible by $p$. Hence the degree of $h$ over zero is always divisible by $p$ and the degree of $f_1$ as a map of $\partial \Delta^{n-1}$ to itself is $1$ modulo $p$. 
\end{proof}

\section{A segment partition problem with choosing nothing}
\label{section:segment}

One particular setting, which we borrow from \cite{segal2018,meunier-zerbib2018}, is when a point $(t_1,\ldots, t_n)\in \Delta^{n-1}$ is interpreted as a partition of a unit segment, in this case different points of the simplex in fact give the same partition. More precisely, in the vector $(t_1,\ldots, t_n)$ we may move zero coordinates of this vector to any position, only keeping the order of positive coordinates, the actual partition of the segment will be the same. Hence the preferences of the players have to follow these permutations, which gives us a modification of the equivariance assumptions.

\subsection{Pseudo-equivariance assumptions}
\label{section:pseudo-equivariant}

Now it is natural to introduce \emph{the segment partition problem with the possibility of choosing nothing} so that preferences are in accordance with the above described identifications. Those identifications can be described by identifying the proper faces of $\Delta^{n-1}$ by linear maps. Those maps $\sigma_{FGZ} : F \to G$ may be viewed as permutations of the coordinates $\sigma_{FGZ} : \Delta^{n-1} \to \Delta^{n-1}$ of the simplex, that move the nonzero coordinates of a face $F$ to the nonzero coordinates of another face $G$ preserving their order, and move the zero coordinates of a face $F$ to zero coordinates of a face $G$ with an arbitrary bijection, which we denote by $Z$. In particular, for given $F$ and $G$ of dimension $k$ there are $(n-k-1)!$ bijections $Z$. The possibility to permute the zero coordinates arises because those permutations do not change the actual partition of the segment. 

We also assume that a player is not allowed to take nothing in the presence of $n$ non-empty parts, otherwise we would have to drop a part, as we did in the previous section. This keeps the covering property, for every $j=1,\ldots,n$,
\[
\Delta^{n-1} = \bigcup_{i=1}^n A_{ij}
\]
and allows, as in the proof of Theorem~\ref{theorem:gale}, to pass to the continuous map $f : \Delta^{n-1}\to \Delta^{n-1}$ setting. In terms of the continuous map, we then have the restrictions 
\begin{equation}
\label{equation:pseudo-eq}
f \circ \sigma_{FGZ} = \sigma_{FGZ} \circ f\quad\text{valid on the face}\quad F. 
\end{equation}
Let us clarify these relation. For given $F,G,Z$ this relation is only applied to points $x\in F\subset\Delta^{n-1}$. The image $\sigma_{FGZ}(x)$ on the left hand side then belongs to $G$, and then $f$ applies to it. On the right hand side we first apply $f$ to $x$ to obtain a point in the simplex that need not belong to any specific facet; after that we apply $\sigma_{FGZ}$ defined as a permutation, taking its $Z$ part into account. 

Note that this setting resembles a certain equivariance assumption on the map $f$, at least on the boundary of $\Delta^{n-1}$. But this is not quite that, because the permutations $\sigma_{FGZ}$ do not constitute a group and the commutation restrictions \eqref{equation:pseudo-eq} are only applied for points lying on the facet $F$. For briefness, let us call a continuous $f:\Delta^{n-1}\to\Delta^{n-1}$ satisfying the commutation restrictions \eqref{equation:pseudo-eq} \emph{pseudo-equivariant}.

Of course, we need to explain, how to pass from sets to continuous functions in the pseudo-equivariant case. Relations \eqref{equation:pseudo-eq} in terms of closed sets $A_{ij}$ read
\begin{equation}
\sigma_{FGZ} \left( A_{ij} \cap F \right) = A_{\sigma_{FGZ}(i)j}\cap G,
\end{equation}
which assumes the form \eqref{equation:pseudo-eq}, when we pass from the closed sets $A_{ij}$ to their upper semicontinuous indicator functions $\chi_{ij} = \chi_{A_{ij}}$. If we approximate the indicator functions by continuous functions without due caution, the assumptions \eqref{equation:pseudo-eq} may fail at a point $x$ in a face $F$, because during the approximation of the $\chi_{ij}$ by continuous functions $f_{ij}$ the values $f_{ij}(\sigma_{FGZ}(x))$ may be influenced by nearby points not belonging to $F$ and not subject to the relation \eqref{equation:pseudo-eq}. 

In order to pass to continuous functions correctly, we put our $\Delta$ into a slightly enlarged concentric simplex $\widetilde\Delta$, and first extend the upper semicontinuous indicator functions $\chi_{ij}$ to $\widetilde\Delta$ by composing them with the metric projection $\pi : \widetilde\Delta\to\Delta$, $\chi_{\widetilde A_{ij}} = \chi_{ij}\circ \pi$. This does not affect the existence of solutions for the partition problem, but allows us to conclude that \eqref{equation:pseudo-eq} will now hold not only on a face $\widetilde F\subset \widetilde\Delta$, but also in some $\epsilon$-neighborhood of $\widetilde F$, for some $\epsilon>0$, because the new $\widetilde F$ projects to the corresponding original $F$ along with its neighborhood. After that we choose a single $\epsilon>0$ for all faces, take continuous functions
\[
g_{ij}(x) = \max\left\{1 - \frac{\dist(x, \widetilde A_{ij})}{\epsilon}, 0 \right\},
\]
and then normalize
\[
f_{ij}(x) = \frac{g_{ij}(x)}{\sum_{i'} g_{i'j}(x)}.
\]
The relations \eqref{equation:pseudo-eq} will hold for such functions on respective faces of $\widetilde\Delta$, since they only depend on the behavior of $\widetilde A_{ij}$ in the $\epsilon$-neighborhood of $x$.

\subsection{A positive solution when $n$ is a prime power}

The arguments in the previous section reduce the segment partition problem with the possibility of choosing nothing to proving that a pseudo-equivariant map $f : \Delta^{n-1}\to\Delta^{n-1}$ sends some point to the center of the simplex. 

\begin{theorem}
\label{theorem:prime-power}
When $n=p^\alpha$, for a prime $p$, any pseudo-equivariant map $f : \Delta^{n-1}\to\Delta^{n-1}$ in the sense of \eqref{equation:pseudo-eq} hits the center $c\in\Delta^{n-1}$. 
\end{theorem}

\begin{proof}
We fix $n=p^\alpha$ and omit it from the notation where appropriate. Like in the proof of Theorem~\ref{theorem:equivariant-gale}, in order to prove what we need, it is sufficient to show that $f(\partial \Delta)$ either has nonzero linking number with the center of $\Delta$, or touches the center. If it touches the center then the problem is solved; hence assume that the center is not touched by $f(\partial \Delta)$ and study the linking number.

Similar to the proof of Theorem \ref{theorem:equivariant-gale}, in order to have information about the linking number we start with the identity map $f_0 : \Delta\to\Delta$, which is pseudo-equivariant and has the linking number of $f(\partial \Delta)$ with the center equal to $1$. It then remains to show that once we deform this $f_0$ to arbitrary $f_1$ pseudo-equivariantly, the linking number may only change by a multiple of $p$, thus remaining always nonzero.

The linking number changes when a point in the boundary $x\in \partial \Delta$ passes through the center $c$ under a pseudo-equivariant homotopy $h_t$ with parameter $t$. If $x$ lies in the relative interior of a $k$-dimensional face $F$ of $\Delta$ then we may apply the relations \eqref{equation:pseudo-eq} to $x$ with different $G$ and $Z$. Those relations show that in total $\binom{n}{k+1}$ images $h_t(\sigma_{FG}(x))$ pass through $c$ together with $x$. Let us call the points $\sigma_{FGZ}(x)$ for different $G$ of dimension $k$ (they do not depend on $Z$) the \emph{pseudo-orbit} of $x$.

The change in the linking number corresponds to the sum of mapping degrees of the homotopy
\[
h : \partial \Delta\times [0,1] \to \Delta
\]
at the points of $h^{-1}(0)$. To make the argument correct, we may assume $h$ piece-wise linear and perturb it generically, keeping the pseudo-equivariance conditions. For any point $x$ in the relative interior of a face $F$, the relations \eqref{equation:pseudo-eq} restrict the image $h(x,t)$ to the linear span of $F$ (``linear'' in the sense that we put the origin to the center of $\Delta$), which has dimension no less than $F\times [0,1]$. Hence, exactly as in the proof of Lemma \ref{lemma:finite-to-one}, a generic pseudo-equivariant PL map $h$ has the property that the preimage of the center under $h$ is a discrete point set, consisting of several pseudo-orbits; and the local mapping degrees are correctly defined.

If we had an equivariance for $h$ under a group action making this pseudo-orbit a real orbit, and permuting their neighborhoods in $\partial\Delta$ accordingly, then we would have that the change in the linking number equals $\binom{n}{k+1}$ times an integer, which would do the job since such a binomial coefficient is divisible by $p$ when $n=p^\alpha$. But we only have pseudo-equivariance in \eqref{equation:pseudo-eq}, whose equations with $\sigma_{FGZ}$ are only applied on the respective face $F$.

In order to use the pseudo-equivariance correctly, we notice that any point of the considered pseudo-orbit belongs to $n-k-1$ facets of $\Delta$ and its disk neighborhood in $\partial \Delta$ splits into $n-k-1$ parts. Some of those parts of neighborhoods of the points in the pseudo-orbits are identified by the maps $\sigma_{\Delta_i\Delta_j}$, corresponding to pairs of facets (the bijection $Z$ in this case is always unique). Since we have $n$ facets in total, we in fact split the parts of neighborhoods of the pseudo-orbit to identified $n$-tuples. 

We may calculate the sum of mapping degrees of $h$ over the pseudo-orbit (or over all points mapped to the center of $\Delta$) by choosing a radially symmetric differential form $\nu\in \Omega^{n-1}(\Delta)$ supported near the center of $\Delta$ with unit integral and integrating its pull-back over the neighborhoods of our pseudo-orbit points. The integration is possible, since we consider a piece-wise linear $h$. We essentially use the mapping degree formula (see \cite[page 188]{guillemin-pollack2010}, for example)
\[
\int_{\partial \Delta\times [0,1]} h^*\nu = (\deg h) \int_{\Delta} \nu = \deg h,
\]
taking in account that the image of the boundary of $\partial \Delta\times [0,1]$ does not hit the support of $\nu$, the neighborhood of the center of $\Delta$. From the assumption that the piece-wise linear map $h$ is in general position, the integral on the left hand side is in fact the integral over neighborhoods of points in the preimage of the center of $\Delta$, if we choose the support of $\nu$ sufficiently small. Hence we assume that we are now studying one pseudo-orbit of such points and integrate over a union of their neighborhoods, split into parts, in order to estimate the corresponding part of the mapping degree of $h$.

Once we split the neighborhoods into parts according to the facets of $\partial\Delta$, we may integrate $h^*\nu$ over every part $P$ of a neighborhood of a point in the pseudo-orbit to obtain a \emph{partial mapping degree} of $P$,
\[
\deg_P h = \int_P h^*\nu.
\]
Here we assume that the parts of neighborhoods $P$ are oriented according to the orientation of $\partial\Delta$. Then the sum over all parts of neighborhoods will be the degree of $h$ in the neighborhood of the pseudo-orbit in question. Note that a partial mapping degree is a real number, not necessarily an integer. The identifications $\sigma_{\Delta_i\Delta_j}$ show that among the numbers $\deg_P h$ obtained by such integration some are equal, the whole collection of these partial mapping degrees in fact split into $n$-tuples of equal real numbers. Those equalities appear with no sign, since $\nu$ is radially symmetric and only changes its sign according to the sign of a permutation of coordinates, which occurs simultaneously in the domain, where the orientation of $\partial\Delta$ also changes according to the sign of the permutation, and in the image of $h$. 

Another relation for the partial mapping degrees $\deg_P h$ is that the sum of partial mapping degrees over the parts of the neighborhood of every point in the pseudo-orbit is an integer, possibly depending on the point, the ordinary local mapping degree. 

We want to use the two types of equalities described above and show that the sum of all partial mapping degrees for the pseudo-orbit in question is an integer divisible by $p$. After the summation over all pseudo-orbits going to the center of $\Delta$ under $h$, this will show that the full mapping degree of $h$ is divisible by $p$ and therefore the degree of $f|_{\partial \Delta}$ as a map from $\partial \Delta$ to $\Delta\setminus \{c\} \sim \partial\Delta$ is always $1$ modulo $p$, as it is for the identity map $f_0$. From this we can conclude that $f$, as a map $\Delta\to\Delta$, always touches the center of the simplex. 

Let us introduce some notation in order to work with partial mapping degrees and their sum. Consider a point $x$ in the pseudo-orbit, describe its \emph{kind} by the sequence $[y_1, \ldots, y_{k+2}]$, where $y_i$ is the number of zero coordinates between the $(i-1)$th and $i$th nonzero coordinates of $x$. More precisely, if $x_{i_1}, \ldots, x_{i_{k+1}}$ are the nonzero coordinates of $x$ then the kind of $x$ is $[i_1-1, i_2-i_1-1, \ldots, i_{k+1}-i_k-1, n-i_{k+1}]$. For example, the point $(0, x_2, 0, 0, x_5)$ will have the kind $[1, 2, 0]$. For any sequence $y_1, \ldots, y_{k+2}$ of non-negative integers summing up to $n-k-1$ there corresponds a unique point of kind $[y_1, \ldots, y_{k+2}]$ in the pseudo-orbit of a given point $x$ from a relative interior of a $k$-dimensional face of the simplex. Hence we may use the kinds to enumerate points in a pseudo-orbit.

Let $P$ be a part of the neighborhood of a point of the kind $[y_1, \ldots, y_{k+2}]$ in the facet given by $t_i=0$. The $i$th coordinate of the point is $0$ and there is some $y_j$ to which it corresponds. Hence $P$ is uniquely described by $[y_1, \ldots, y_{k+2}]$ with sum $n-k-1$ and the choice of the index $j$ of the position of zero. We may view the points of $P$ as $k+1$ big coordinates, $n-k-2$ small coordinates (which were zero for original pseudo-orbit points in $k$-faces), and one zero. The sequence $[y_1, \ldots, y_{j-1}, y_j-1, y_{j+1}, \ldots, y_{k+2}]$ then describes the positions of small coordinates among big coordinates and ignores zero. The identifications of $n$ such parts of neighborhoods in a pseudo-orbit corresponds to inserting zero into arbitrary position of a given sequence of big and small coordinates; therefore it is natural to call $[y_1, \ldots, y_{j-1}, y_j-1, y_{j+1}, \ldots, y_{k+2}]$ the \emph{kind} of a pseudo-orbit of parts of neighborhoods. Then to each sequence $y_1, \ldots, y_{k+2}$ of non-negative integers summing up to $n-k-2$ there corresponds a unique part of neighborhood kind.

Moreover, we denote by $\deg[y_1, \ldots, y_{j-1}, y_j-1, y_{j+1}, \ldots, y_{k+2}]$ the partial mapping degree of any part of a neighborhood of the given kind, this degree indeed only depends on the kind. In order to prove the theorem, we need to show that the sum of all such degrees, multiplied by $n$, is an integer divisible by $p$. We split this sum into several parts, for any integer $0\leq r\leq n-k-2$, put 
\[
S_r=\sum_{r+y_2 + \dots + y_{k+2}=n-k-2} \deg[r, y_2, \ldots, y_{k+2}],
\] 
and put $S_{-1}=0$ for consistency. What we need to prove then translates to 
\begin{equation}
\label{eq:goal}
n\sum_{r=0}^{n-k-2} S_r \equiv 0 \mod p.
\end{equation}

Summing up the partial mapping degrees in the neighborhood of the point of the kind $[y_1, \ldots, y_{k+2}]$ we get
\begin{equation}
\label{eq:ngb}
\sum_i y_i \deg[y_1, \ldots, y_i-1, \ldots, y_{k+2}] \in {\mathbb Z}.
\end{equation}

Summing up formulas of \eqref{eq:ngb} for different kinds with $y_1=r$ we get
\begin{equation}
\label{eq:sumr}
rS_{r-1} + (n-r-1)S_r\in{\mathbb Z}.
\end{equation}
Indeed, each $\deg[r-1, y_2, \ldots, y_{k+2}]$ contributes with coefficient $r$ in \eqref{eq:ngb} for the neighborhood of the point of the kind $[r, y_2, \ldots, y_{k+2}]$. And each $\deg[r, y_2, \ldots, y_{k+2}]$ contributes with coefficient $y_2+1$ in \eqref{eq:ngb} for the neighborhood of the point of the kind $[r, y_2+1, \ldots, y_{k+2}]$, with the coefficient $y_3+1$ in \eqref{eq:ngb} for the neighborhood of the point of the kind $[r, y_2, y_3+1, \ldots, y_{k+2}]$, and so on. Its total contribution then is 
\[
(y_2 + 1) + \dots + (y_{k+2} + 1),
\] 
which is equal to $n-k-2 - r + (k+1)=n-r-1$.

Let us prove by induction that 
\begin{equation}
\label{eq:sumr_ind}
(r+1)\binom{n-1}{r+1}S_r\in{\mathbb Z}.
\end{equation}
The base $r=0$ of induction follows from \eqref{eq:sumr} with $r=0$. Suppose we have proved \eqref{eq:sumr_ind} for some $r$. Writing \eqref{eq:sumr} for $r+1$, we get
\[
(r+1)S_{r} + (n-r-2)S_{r+1}\in{\mathbb Z}.
\]
Multiply by $\binom{n-1}{r+1}$ to get
\[
(r+1)\binom{n-1}{r+1}S_{r} + (n-r-2)\binom{n-1}{r+1}S_{r+1}\in{\mathbb Z}.
\]
By the induction assumption, we have
\[
(n-r-2)\binom{n-1}{r+1}S_{r+1}\in{\mathbb Z}.
\]
Substituting $\binom{n-1}{r+1}=\frac{r+2}{n-r-2}\binom{n-1}{r+2}$, we get the desired result
\[
(r+2)\binom{n-1}{r+2}S_{r+1}\in{\mathbb Z}.
\]

Since $n=p^{\alpha}$ is a prime power, then all digits of $n-1$ in $p$-adic notation are $p-1$. Hence, by the Lucas theorem \cite{lucas1878} we get that $\binom{n-1}{r+1}$ is not divisible by $p$. This means that $(r+1)\binom{n-1}{r+1}$ is not divisible by $p^{\alpha}$ for all $0\leq r\leq n-k-2$, since $r$ is not divisible by $p^\alpha$. Therefore, the least common multiple $m$ of the numbers $(r+1)\binom{n-1}{r+1}$ for all $0\leq r\leq n-k-2$ is also not divisible by $p^{\alpha}$.

From \eqref{eq:sumr_ind} we conclude that
\[
m\sum_r S_r\in{\mathbb Z}.
\]
For each kind of a neighborhood there are exactly $n$ partial neighborhoods of this kind, so we also know that 
\[
n\sum_r S_r = p^\alpha \sum_r S_r\in{\mathbb Z}.
\]
Hence, $n\sum _r S_r$ is divisible by $\frac{n}{\mathrm{gcd}(n, m)}$, which in turn is divisible by $p$, because $m$ is not divisible by $n=p^{\alpha}$. This establishes \eqref{eq:goal} and completes the proof.
\end{proof}

\subsection{Counterexamples when $n$ is not a prime power}
\label{section:counterexamples}

As it was shown above, in order to build a counterexample, where the segment partition problem with possibility to choose nothing and no part can be dropped has no solution, it is sufficient to build a pseudo-equivariant map $f : \Delta^{n-1}\to\Delta^{n-1}$ missing the center $c\in\Delta^{n-1}$ and put 
\[
A_{ij} = \left\{t\in \Delta^{n-1} \where \forall i'\ f_i(t)\ge f_{i'}(t) \right\}
\]
independent on the player index $j$.

The first observation is that it is sufficient to have a pseudo-equivariant map $f$ such that the image of the boundary $f(\partial \Delta^{n-1})$ is not linked with the center $c\in\Delta^{n-1}$. Since the homotopy group $\pi_{n-2}\left(\Delta^{n-1}\setminus \{c\}\right)$ is $\mathbb Z$, the possibility to (re)extend $f$ continuously to the interior of the simplex $\Delta^{n-1}$ is fully governed by the linking number and any such continuous extension does not violate the pseudo-equivariance relations \eqref{equation:pseudo-eq}, because the relations are only applicable on the boundary of the simplex.

The second observation is that it is sufficient to find a continuous map $f : \partial \Delta^{n-1}\to \Delta^{n-1}$ having zero linking number of the image with the center of the simplex and equivariant with respect to the action of the full permutation group $\mathfrak S_n$. The full equivariance on the boundary implies the pseudo-equivariance we need, and a continuous extension of $f$ to the interior of the simplex is possible provided the linking number is zero. 

In what follows we will switch between the two points of view: To find $f : \partial \Delta^{n-1}\to \Delta^{n-1}$ with zero linking number with the center is the same as to find $f : \partial \Delta^{n-1}\to \partial \Delta^{n-1}$ with zero mapping degree. In order to see these are the same just compose $f$ with a central projection from the center of the simplex to have its image contained in the boundary of the simplex; and note that such a projection preserves equivariance and pseudo-equivariance.

One counterexample is in fact a counterexample to Theorem \ref{theorem:equivariant-gale}.

\begin{theorem}
\label{theorem:n-odd}
If $n$ is odd and not a prime power then there exists an $\mathfrak S_n$-equivariant continuous $f : \partial \Delta^{n-1}\to\Delta^{n-1}$ of zero linking number with the center of $\Delta^{n-1}$.
\end{theorem}
\begin{proof}
We fix $n$ and omit $n$ from the notation where appropriate. We will start with the identity $f_0 : \partial \Delta \to \partial \Delta$, considered also as the inclusion $\partial \Delta \to \Delta$. It definitely has degree $1$ and we are going to modify it equivariantly so that its mapping degree will become $0$.

A modification will consist in taking a dimension $k$, all the centers of the $k$-dimensional simplices $c_1,\ldots,c_N$, $N=\binom{n}{k+1}$, and pulling the images $f(c_i)$ to the center of $\Delta$ (along with pulling their neighborhoods continuously and equivariantly). When the images $f(c_i)$ cross the origin, the linking number of $f(\partial \Delta)$ will change by either $+1$ or $-1$ at every point, and by $\pm\binom{n}{k+1}$ in total. 

Of course, in such a modification the sign $+$ or $-$, at first glance, is fixed. But we may not only pull a point $c_1$ towards the origin, but also flip the mapping derivative image of the tangent space $T_{c_1}F$ to the $k$-face $F$ containing $c_1$ on the way. Such a flip commutes with the stabilizer of $c_i$ in the permutation group and can therefore be extended equivariantly to the neighborhood of the orbit $\{c_i\}$. Moreover, when $k$ is odd, this flip will change the sign of the crossing and therefore we will be able to choose the sign of the modification by applying or not applying the flip before the crossing. See the details of this pulling and flipping moves, for $n=3$, in Figures \ref{figure:pulling} and \ref{figure:pulling-all}.

When $k$ is even, the flip does not change the sign of the crossing, hence we are only able to make one crossing, and when we pull the point $c_1$ (and equivariantly its orbit) back through the center of $\Delta$, we just make the opposite crossing and return to where we started from in terms of the linking number. When $k$ is odd, we have much more freedom. We may pull the images $f(c_i)$ and their neighborhoods to the center $c\in\Delta$ once again and once again choose the sign of the crossing using or not using the equivariant flip before the crossing. In total, for odd $k$, this allows us to change the linking number by any multiple of $\binom{n}{k+1}$, positive or negative. Figure \ref{figure:pulling-twice} shows how to make two successive changes of the linking number in the same direction.

\begin{figure}[ht]
\center
\includegraphics[width=100mm]{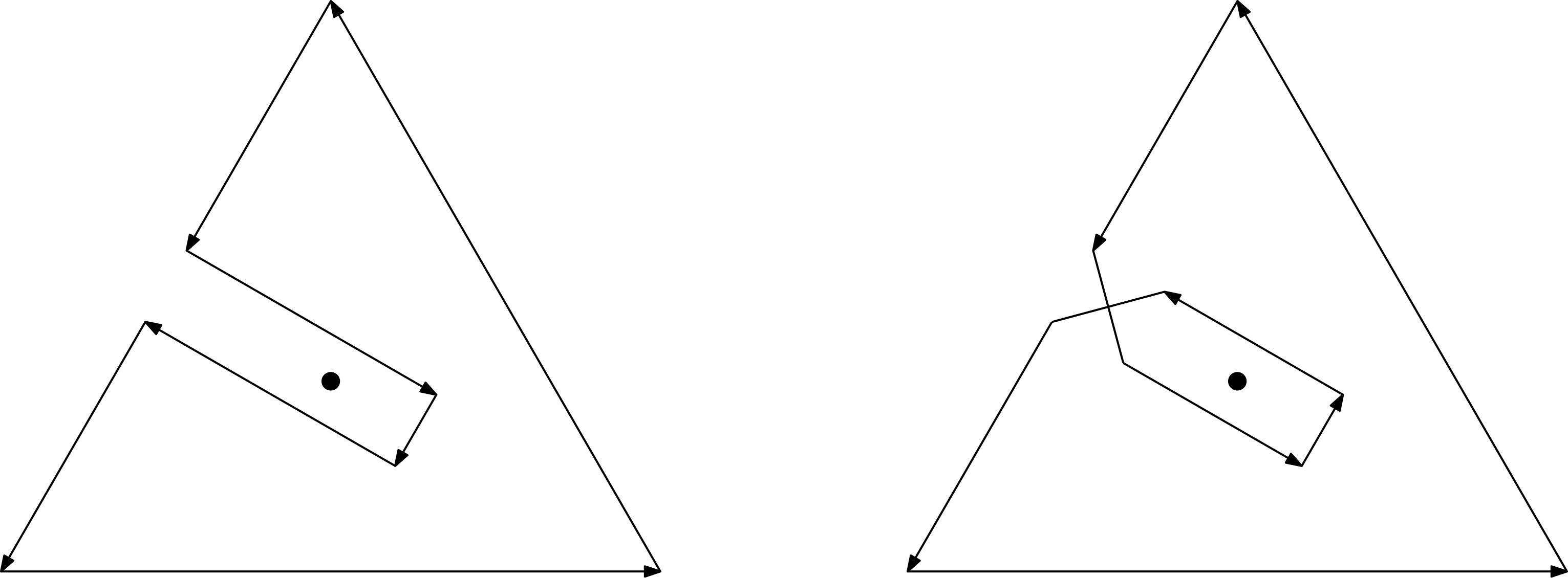}
\caption{Pulling one point towards the center with/without a flip of signs.}
\label{figure:pulling}
\end{figure}

\begin{figure}[ht]
\center
\includegraphics[width=100mm]{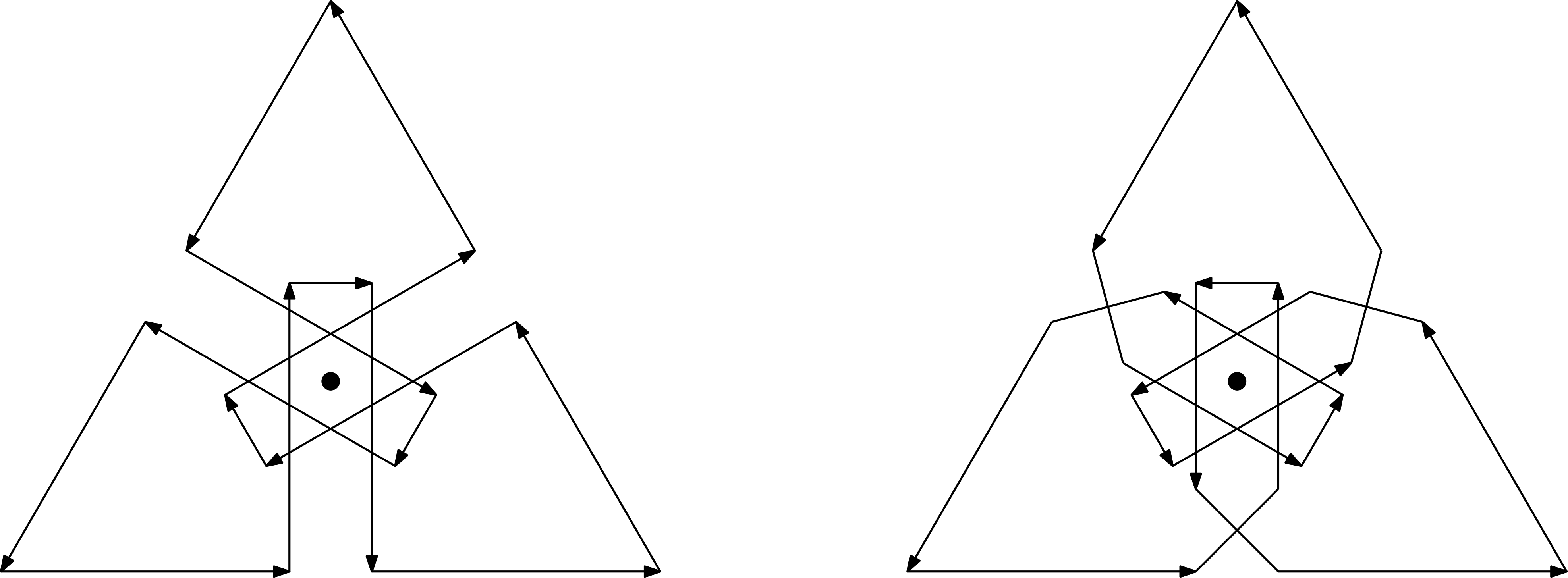}
\caption{Pulling an orbit of points towards the center.}
\label{figure:pulling-all}
\end{figure}

\begin{figure}[ht]
\center
\includegraphics[width=60mm]{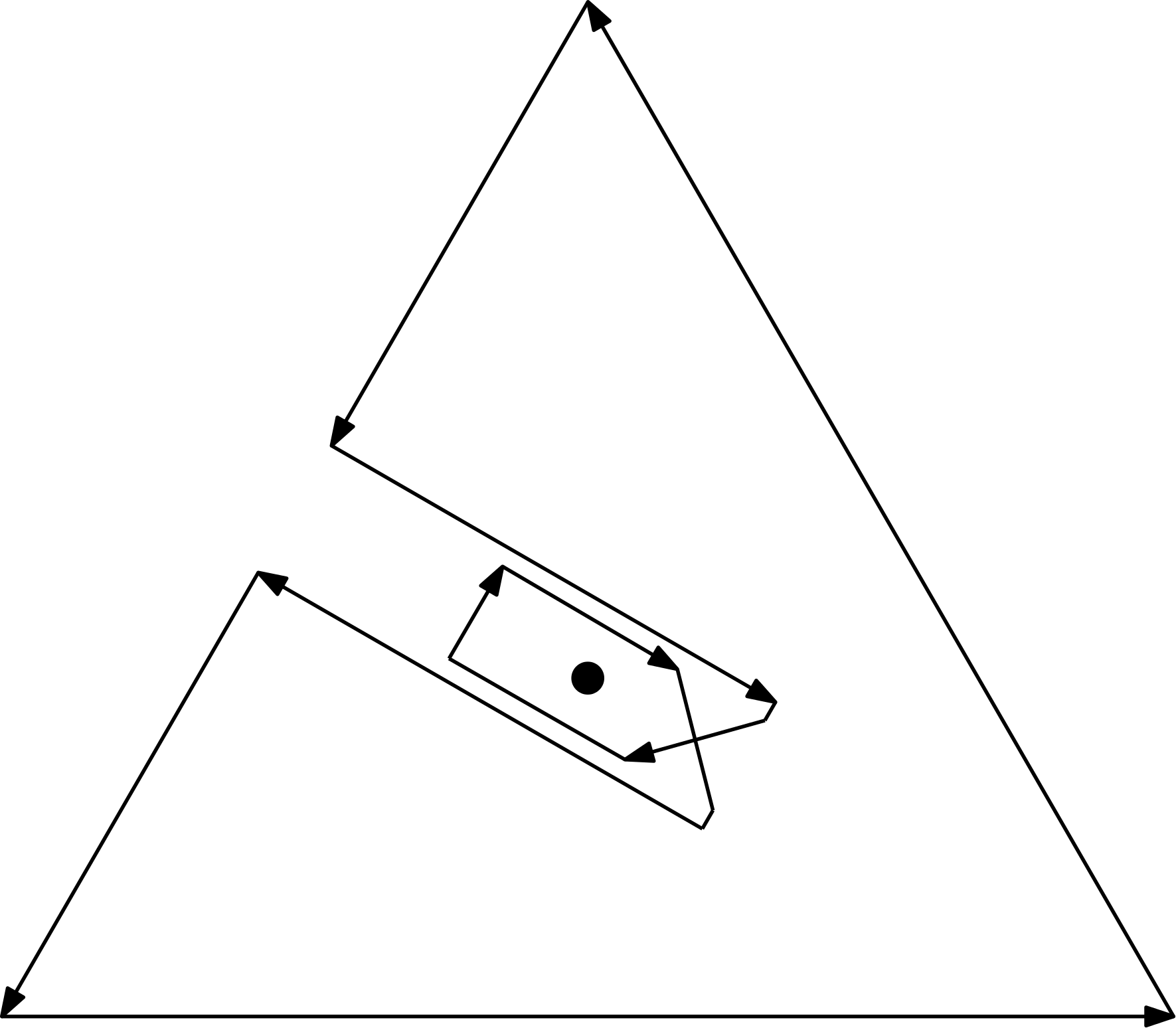}
\caption{Pulling a point towards the center and then pulling it back with a flip. The other points in the orbit are not shown.}
\label{figure:pulling-twice}
\end{figure}

Recall Ram's theorem \cite{ram1909} (or the Lucas theorem \cite{lucas1878} that we have already used) that asserts that there exist integers $x_1,\ldots,x_{n-1}$ such that
\[
x_1 \binom{n}{1} + x_2 \binom{n}{2} + \dots + x_{n-1}\binom{n}{n-1} = -1,
\]
provided $n$ is not a prime power. Note that in our case $n$ is not a prime power. 

Moreover, $n$ is odd and therefore, in view of the symmetry $\binom{n}{k+1} = \binom{n}{n-k-1}$, the set of the binomial coefficients is the same as the set of binomial coefficients with even $k+1$. Hence, if we repeatedly use our moves for odd $k$ with possible flips then by Ram's theorem we will be able to modify the linking number of $f(\partial \Delta)$ with $c$ from $1$ to zero. 
\end{proof}

It remains to handle the case of even $n$, but this is less easy. In the above argument we cannot change the crossing sign for even $k$ and $n-k-2$, in particular, we can add or subtract $\binom{n}{k+1}$ from the linking number, but cannot repeat this operation, since when we move the orbit back to the center of $\Delta$, we just change the linking number back. A flip was really needed in order to have a chance to repeat the change by $\pm\binom{n}{k+1}$ several times in the same direction. In particular, for $n=6$ we failed to produce a $\mathfrak S_6$-equivariant map $\Delta^5\to\Delta^5$ of zero degree by hand.

What we are able to do now, is to do this in the setting of pseudo-equivariance instead of full equivariance. The following result shows that the segment partition problem with the possibility of choosing nothing has no solution if $n$ is not a prime power.

\begin{theorem}
\label{theorem:n-even}
If $n$ is not a prime power then there exists a pseudo-equivariant, in terms of relations \eqref{equation:pseudo-eq}, continuous $f : \partial \Delta^{n-1}\to\Delta^{n-1}$ of zero linking number with the center of $\Delta^{n-1}$.
\end{theorem}
\begin{proof}
We do the same modifications as in the previous proof, but we need to handle the case of even $k$. In view of the relations $\binom{n}{k+1} = \binom{n}{n-k-1}$ we may also assume that $k\ge n/2-1\ge 2$.

Note that, for a $k$-face $F$, any composition of the pseudo-equivariance symmetries $\sigma_{F'G'Z}$ with $F'\supseteq F$ cannot take the face $F$ to itself and induce a non-identity map on it, because all such symmetries preserve the order of the nonzero coordinates. Hence we can choose a direction $v_1\in T_{c_1}F$ (because we only consider faces of positive dimension) in any point $c_1$ in the relative interior of $F$ and we will have the well-defined defined pseudo-orbit $\{c_i\}$ of this point and this direction $v_i\in T_{c_i}F_i$, so that the pseudo-equivariance symmetries permute those points and those directions whenever they are defined on them. 

Now we modify the original identity map $f_0$, we pull the images of the pseudo-orbit $f(c_i)$ towards the center $c$ of $\Delta$ and on the way to the center we flip the tangent space $f_* \left( T_{c_1}F_1 \right)$ along the chosen direction $f_* v_1$, if we need to switch the sign of the crossing. The corresponding flips around every point of the pseudo-orbit $\{f(c_i)\}$ will be made in the pseudo-equivariant fashion, in total allowing us to modify the linking number by $\pm\binom{n}{k+1}$ with a sign we choose. 

It is possible to iterate such steps, moreover, in the absence of the true equivariance we are allowed to choose $c_1\in F$ different from the center of $F$, making every step independent of the other steps. Having the possibility to choose the sign and iterate, in view of Ram's theorem for non-prime power $n$, we can obtain zero linking number.
\end{proof}

\section{A negative result for the equivariant fair partition technique}
\label{section:borsuk-ulam}

\subsection{Failure of a Borsuk--Ulam-type result from the mapping degree}

One general approach to envy-free segment partition problems (or fair partition problems, as in \cite{ahk2014,aak2018}) is to introduce a configuration space $X$ with an action of $\mathfrak S_n$ and a test map $f : X\to \mathbb R^n$ equivariant with respect to the action of $\mathfrak S_n$ on $X$ and its action on $\mathbb R^n$ by permuting the coordinates so that a solution to the problem is a situation when for some $x\in X$ the image $f(x)$ hits the diagonal
\[
D_n = \{(u,u,\dots,u) \in \mathbb R^n\where u\in\mathbb R\}.
\]
Sometimes, a Borsuk--Ulam-type theorem guarantees such a diagonal hit. We now show that Theorem \ref{theorem:n-odd} guarantees that there is no such Borsuk--Ulam-type theorem for certain values of $n$:

\begin{theorem}
\label{theorem:odd-non-pp}
Assume $n$ is odd and not a prime power. Then for any Hausdorff compactum $X$ with a free action of $\mathfrak S_n$ there exists a continuous $\mathfrak S_n$-equivariant map $X\to \mathbb R^n$ not touching the diagonal $D_n\subset\mathbb R^n$.
\end{theorem}
\begin{proof}
Consider the orthogonal decomposition $\mathbb R^n = D_n\oplus W_n$ and the unit sphere $S(W_n)$ in the $(n-1)$-dimensional space $W_n$. It is possible to map the simplex $\Delta^{n-1}$ to $W_n$ equivariantly, just subtracting $1/n$ from every barycentric coordinate, then the radial projection from the origin will identify $\partial \Delta^{n-1}$ with $S(W_n)$ equivariantly.

Theorem \ref{theorem:n-odd} in these terms says that there exists an equivariant map $S(W_n)\to S(W_n)$ of mapping degree $0$. It remains to use Lemma \ref{lemma:zero-degree} below. This lemma gives a $\mathfrak S_n$-equivariant map $X\to S(W_n)$. Composing it with the inclusion $S(W_n)\subset\mathbb R^n$ we obtain an equivariant map from $X$ to $\mathbb R^n$ not touching the diagonal of $\mathbb R^n$.
\end{proof}

Now we present the lemma that (together with its proof) was communicated to us by Alexey Volovikov. It is a particular case of \cite[Lemma~3.9]{bartsch1993}\footnote{$S$ is a $G$-CW complex since it is a sphere of a linear representation of $G$ and hence can be $G$-equivariantly triangulated.}, but we present a short proof of the particular case we need here for completeness.

\begin{lemma}
\label{lemma:zero-degree}
Let $G$ be a finite group and $S$ be a sphere with an action of $G$. If there exists an equivariant map $f : S\to S$ of zero degree then any Hausdorff compactum $X$ with a free action of $G$ has an equivariant map $X\to S$.
\end{lemma}

\begin{proof}
A zero degree map of spheres $S\to S$ is null-homotopic and can be continuously extended to a cone over the sphere $S$. Consider the join $G * S$ as a union of $|G|$ such cones glued together along their bases and extend the map from one cone to all other cones by equivariance with respect to the diagonal action of $G$ on the join, obtaining an equivariant map $g : G*S \to S$. Then take joins of $g$ with identity maps of $G$ and compose them to extend the chain of equivariant maps
\[
\cdots \to G * G * G * S \to G * G * S \to G * S \to S.
\]
Since every component of the join embeds into the join, we may drop $S$ in the domain and eventually have an equivariant map as a composition:
\[
\underbrace{G * G * \dots * G}_N \to \underbrace{G * G * \dots * G}_N * S\to S
\]
for any $N$. 

The join in the domain of the last map is the $(N-2)$-connected $(N-1)$-dimensional approximation $E_N G$ to the classifying space $EG$ of the group $G$. By standard properties of the classifying spaces it follows that, given a Hausdorff compactum $X$ with a free action of $G$, there exists an equivariant map $X\to E_N G$ for sufficiently large $N$, hence there exists an equivariant map $X\to S$ as a composition of $X\to E_N G \to S$.
\end{proof}

\begin{remark}
The assumption on compactness of $X$ in Theorem \ref{theorem:odd-non-pp} is not very restrictive in practical situations, since in most cases the non-compact configurations spaces for fair partition problem are $\mathfrak S_n$-equivariantly homotopy equivalent to their compact models, as it happened in \cite[Theorem~3.13]{bz2014}, for example.
\end{remark}

\subsection{Some remarks and consequences}

We first answer some doubts expressed by an unknown referee of \cite{aks2019} about the novelty of Theorem \ref{theorem:odd-non-pp}. Our Theorem \ref{theorem:odd-non-pp} is similar to, but is not a particular case of \cite[Theorem 3.6]{bartsch1993}. Indeed, \cite[Theorem 3.6]{bartsch1993} takes a group $G$ from a certain class and proves that there exists \emph{some} representation $W$ of $G$, for which there exists a $G$-equivariant map $X \to S(W)$ from any fixed point free $G$-space $X$. In Theorem \ref{theorem:odd-non-pp}, by contrast, we prove that for a specific group $G=\mathfrak S_n$ and a \emph{specific} representation sphere of $G$, $S(W)$, there exists a $G$-equivariant map $X\to S(W)$ from any free $G$-space $X$. The group $G=\mathfrak S_n$ does not satisfy the hypothesis of \cite[Theorem 3.6]{bartsch1993}, because it contains a subgroup (the alternating group) of prime index. The discussion in \cite[the paragraph after Theorem 3.6]{bartsch1993} also hints that our specific $W$ cannot be the one constructed in the proof of \cite[Theorem 3.6]{bartsch1993}, since our $W$ has the property $W^H=0$ whenever a subgroup $H\subset \mathfrak S_n$ acts transitively on the indices $1,\ldots, n$.

Now we briefly outline the open problems and the work that appeared after this paper was published as an Arxiv preprint. Theorems \ref{theorem:equivariant-gale} and \ref{theorem:odd-non-pp} leave the question ``For which $n$ is it possible to have a $\mathfrak S_n$-equivariant map $S(W_n)\to S(W_n)$ of zero degree?'' open in the case when $n$ is not a prime power and is even. The final resolution of this question requires more technicalities and is done in the separate paper \cite{avku2019}.

In terms of the works \cite{ahk2014,bz2014}, the theorems of this section show that the direct approach to fair partition problems does not only fail in terms of the primary cohomology obstruction, but also in terms of higher obstructions, when $n$ is odd and not a prime power. This approach (and its appropriate generalizations) has some particular consequences for the topological Tverberg problem (or, more generally, to van Kampen--Flores-type problems), which are given in \cite{aks2019}.

Theorem \ref{theorem:odd-non-pp} also provides counterexamples to a class of envy-free division problems, where labeled partitions in $n$ nonempty parts ($n$ is odd and not a prime power) are parametrized by a compact polyhedron $X$, on which $\mathfrak S_n$ acts by permutations of the labels, and the preferences of the $n$ players do not depend on the labels. A counterexample is obtained by taking equivariant $f : X\to S(W_n)$ and assigning the preference of any player to the parts with corresponding maximal coordinate $f_i$ of the map; the situation when every part is preferred by some player is then impossible. For a more detailed exposition of this idea, see \cite[Section~2]{avku2019}.

\bibliography{../Bib/karasev}

\bibliographystyle{abbrv}
\end{document}